\documentclass[a4paper,12pt,leqno,english]{article}
\usepackage[utf8]{inputenc}
\usepackage[T1]{fontenc}
\usepackage{babel}

\usepackage{amsmath}
\usepackage{amssymb}

\usepackage{enumitem}
\usepackage{graphicx}
\usepackage{hyperref}
\usepackage{tcolorbox}
\usepackage{pictex}

\usepackage[disable]{todonotes} 
\usepackage{comment}
\usepackage{tikz} 
\usepackage{mathrsfs}

\newcommand{\B}{{\mathbb  B}}
\newcommand{\C}{{\mathbb  C}}

\newcommand{\D}{{\mathbb D}}

\renewcommand{\L}{{\mathcal L}}

\newcommand{\N}{{\mathbb  N}}

\renewcommand{\O}{{\mathcal O}}
\renewcommand{\P}{{\mathcal  P}}
\newcommand{\Q}{{\mathbb  Q}}
\newcommand{\R}{{\mathbb  R}}

\newcommand{\T}{{\mathbb  T}}
\newcommand{\Z}{{\mathbb  Z}}

\newcommand{\set}[2]{\{ #1 \,;\,  #2 \}} 
\newcommand{\scalar}[2]{{\langle#1,#2\rangle}}

\newcommand{\ch}{{\operatorname{ch}}}

\newcommand{\Log}{{\operatorname{Log}}}

\newcommand{\PSH}{{\operatorname{{\mathcal{PSH}}}}}

\renewcommand{\Re}{{\operatorname{Re}}}

\newcommand{\supp}{{\operatorname{supp}\, }}

\newcommand{\Lg}{\mathscr{L}}

\newcommand{\Conv}{\operatorname{Conv}}

\def\vfigura#1#2{
\setbox0\vbox{{
\input #1
}}
\setbox1\vbox{\hbox{\box0}\hbox{{\obeylines #2}}}
\dimen0 = -\ht1
\advance\dimen0 by-\dp1
\dimen1 = \wd1
\dimen2 = -\dimen0
\divide\dimen2 by\baselineskip
\count100 = 1
\advance\count100 by\dimen2
\advance\count100 by1
\box1
\hangindent\dimen1
\hangafter=-\count100
\vskip\dimen0
}

\numberwithin{equation}{section}

\newtheorem{theorem+}           {Theorem}      [section]
\newtheorem{definition+}  [theorem+]  {Definition}
\newtheorem{lemma+}  [theorem+]  {Lemma}
\newtheorem{corollary+}  [theorem+]  {Corollary}
\newtheorem{proposition+}  [theorem+]  {Proposition}
\newtheorem{example+}  [theorem+]  {Example}
\newtheorem{problem+}  [theorem+]  {Problem}
\newtheorem{remark+}  [theorem+]  {Remark}

\newenvironment{theorem}{\begin{theorem+}\sl}{\end{theorem+}\rm}
\newenvironment{definition}{\begin{definition+}\rm}{\end{definition+}\rm}
\newenvironment{lemma}{\begin{lemma+}\sl}{\end{lemma+}\rm}
\newenvironment{corollary}{\begin{corollary+}\sl}{\end{corollary+}\rm}
\newenvironment{proposition}{\begin{proposition+}\sl}{\end{proposition+}\rm}

\newenvironment{proof}{\medbreak\noindent{\bf  Proof:}\rm}{\hfill$\square$\rm}

\title{{\Large \bf 
Holomorphic approximation by polynomials\\ with exponents restricted to a convex cone
}}
\author{Álfheiður Edda Sigurðardóttir}
\date{{}} 
\begin{document}
\maketitle

\begin{abstract} \noindent
We study the approximation of holomorphic functions of several complex variables by the ring $\P^S(\C^n)$ of polynomials whose exponents are restricted to a convex cone $\R_+S$ for some compact convex $S\in \R^n_+$.
We show a version of the Runge-Oka-Weil Theorem on approximation by these subrings on compact  subsets of $\C^{*n}$ that are convex with respect to $\P^S(\C^n)$. 
We show a sharper result on rotationally symmetric compact sets. The tools used are Hörmander's $L^2$-theory and Siciak-Zakharyuta functions $V^S_K$ associated to $S$. We provide a formula for $V^S_K$ when $K$ is a rotationally symmetric compact subset of $\C^{*n}$.

\end{abstract}

\section[Introduction]{Introduction}
\label{SZsec:01}

Polynomials in $\C^n$ whose exponents are restricted to a convex cone $\Gamma\subset\R^n_+$ form a subring of the polynomials. We will regard cones $\Gamma$ that are the scaling of some compact convex  $S\subset \R^n_+$ with $0\in S$. Throughout this article, $S$ will denote such a set. Let the space ${\mathcal P}^S_m(\C^n)$ consist  of  all polynomials 
$p$ 
of the form $$p(z)=\sum_{\alpha\in (mS)\cap \N^n} a_\alpha z^\alpha, \qquad z\in \C^n.$$
Then $\P^S(\C^n)= \bigcup_{m\in \N}\P^S_m(\C^n)$ is the ring of the polynomials whose exponents are restricted to the cone  $\Gamma=\R_+S$. Note that if $p_1\in \P^S_{m_1}(\C^n)$ and $p_2\in \P^S_{m_2}(\C^n)$, then $p_1p_2\in \P^S_{m_1+m_2}(\C^n)$. Let $\Sigma$ denote the standard unit simplex, that is the convex hull of $0$ and the unit basis $\{e_1,\dots,e_n\}$. Then $p\in \mathcal{P}^\Sigma_m(\C^n)$ if and only if $p$ is of degree $\leq m$. 
Note that $\P^S(\C^n)$ contains all the holomorphic polynomials of $n$ variables if and only of $S$ is a neighborhood of 0 in $\R^n_+=(\R_+)^n$.

\medskip

We study the uniform approximation of holomorphic functions on compact sets by polynomials from polynomial spaces of the form $\P^S(\C^n)$. A motivation for this study comes from \cite{Tre:2017}, where it is shown that $S=\set{x\in \R^n_+}{\|x\|_2\leq 1}$ provides better approximation of holomorphic functions in certain neighborhoods of the unit hypercube $K=[-1,1]^n$ by polynomials from $\P^S_m(\C^n)$ than $S=\Sigma$ would. Better in this context means that the ratio between $\inf\set{\|f-p\|_K}{p\in \P^S_m(\C^n)}$ and the dimension of $\P^S_m(\C^n)$ decays faster as $m$ tends to infinity.

The speed of such approximation by polynomial classes $\P^S(\C^n)$ is studied in more generality in \cite{BosLev:2018}, with a generalization of the Bernstein-Walsh-Siciak theorem for a family of sets $S$ that all are a neighborhood of 0 in $\R^n_+$. 
In order to generalize that theorem for more general convex $S$ in the subsequent work \cite{MagSigSno:2023}, where $\P^S(\C^n)$ may be a proper subset of the polynomials on $\C^n$, we first shall study when approximation is possible in the first place. That is the goal of this current paper.

\medskip

By the Runge-Oka-Weil theorem,  a holomorphic function defined in a neighborhood of the polynomial hull $\widehat{K}$ of compact $K\subset \C^{n}$ can be approximated uniformly on $K$ to arbitrary precision by polynomials. To generalize that theorem, we need a notion analogous to the polynomial hull. We define the \emph{$S$-hull} of a compact set $K\subset \C^n$ by
\begin{equation*}\label{S-hull}
    \widehat{K}^S=\{z\in X\,;\, |p(z)| \leq\|p\|_K\text{ for all }p\in \P^S(\C^n)\}.
\end{equation*}
and we say that $K$ is \emph{$S$-convex} if $\widehat K^S=K$. 
The main theorem of this paper is a generalization of the Runge-Oka-Weil theorem on $S$-convex compact sets.

\begin{theorem}\label{bigTheorem}
    Let $S\subset \R^n_+$  be compact, convex and with $0\in S$. Let $K\subset \C^{*n}=(\C^*)^n$ be compact and  $S$-convex. 
    If $f$ is holomorphic in a neighborhood of $K$, then $f$ can be uniformly approximated on $K$ by polynomials from $\P^S(\C^n)$.
\end{theorem}

\medskip

Note that in Theorem \ref{bigTheorem}, we require $K$ to not intersect the axis hyperplanes. The problem that arises if $K$ does intersect the axis can be summarized by the following example. Let $S\subset \R^2$ be the convex hull of $(0,0)$, $(1,0)$ and $(1,1)$. Then the poly\-nomials in $\P^S(\C^2)$ are all of the form $\text{\emph{constant}}+z_1p(z_1,z_2)$ where $p$ is a polynomial of two variables.  If $K$ contains more than one point in the hyperplane defined by $z_1=0$, then the holomorphic functions in a neighbourhood of $K$ cannot possibly be arbitrarily well approximated by polynomials from $\P^S(\C^2)$. Incidentally, in this example, such a set $K$ cannot be both compact and $S$-convex, because the $S$-hull of a single point ${(0,z_2)}$ is the whole $z_1=0$ hyperplane. Unlike the polynomial hull, an $S$-hull of a compact set is not necessarily compact. It would be interesting to find out if Theorem \ref{bigTheorem} is true for all $K\subset \C^n$ that are compact and $S$-convex. 

 \medskip

If $K\subset \C^n$ is any compact set and $f$ is a locally bounded complex valued function 
 defined on a neighborhood of $\widehat K ^S$ (not necessarily compact) and $f$ can be uniformly approximated by $\P^S(\C^n)$ on compact subsets of $\widehat K ^S$, then $f$ has to be bounded on $\widehat{K}^S$, and in fact $\|f\|_{\widehat{K}^S}=\|f\|_K$. Indeed, 
    for any $\varepsilon>0$ and $z_0\in \widehat K ^S$ 
    there is a $p\in \P^S(\C^n)$ such that $\|p-f\|_{K\cup\{z_0\}}<\varepsilon$. Denoting $C=\|f\|_K$, then $\|p\|_K<C+\varepsilon$. Since $z_0\in \widehat{K}^S$ we have $|p(z_0)|\leq\|p\|_K<C+\varepsilon$. Therefore $|f(z_0)|<C+2\varepsilon$. As this holds for any $\varepsilon>0$ and any $z_0\in \widehat{K}^S$, we can conclude that $\|f\|_{\widehat{K}^S}=C$.

\medskip

We have just presented a necessary condition for any generalization of the Runge-Oka-Weil Theorem for approximation by  $\P^S(\C^n)$, that  $f$ must satisfy $\|f\|_{\widehat{K}^S}=\|f\|_K$. 
This necessary condition is sufficient in the specific case when the compact set $K$ is rotationally symmetric in each variable, that is $(\zeta_1z_1,\dots, \zeta_nz_n)\in K$ for all $z\in K$ and $\zeta$ in the unit torus $\T^n$. We will call such a set a \emph{Reinhardt set}. We arrive at the next theorem. 

\medskip

\begin{theorem}\label{approx_on_E}   Let $S\subset \R^n_+$  be compact, convex and with $0\in S$. Let  $K$ be a compact Reinhardt set, let $X$ be a neighborhood of $\widehat{K}^S$ and let $f\in \O(X)$. Then the following are equivalent:
\begin{enumerate}[label=(\roman*)]
    \item $f$ can be approximated uniformly on $\widehat{K}^S$ by polynomials from $\P^S(\C^n)$.
    \item $f$ can be approximated uniformly on all compact subsets of $\widehat{K}^S$ by polynomials from $\P^S(\C^n)$. 
    \item $\|f\|_{\widehat{K}^S}=\|f\|_K$.
    \item $f$ is bounded on $\widehat{K}^S$.
    \item  There exists a power series of the form $h(z)= \sum_{\alpha\in \R_+S} a_\alpha z^\alpha$ that is convergent on $X$ such that $f=h$ on $\widehat{K}^S$.
\end{enumerate}
\end{theorem}

Condition (i) could then be phrased as so:  $f$ is in the closure of $\P^S(\C^n)$ in the uniform topology on $\mathcal C_b(\widehat{K}^S)$, where  $\mathcal C_b(\widehat{K}^S)$ is the space of continuous bounded functions on $\widehat{K}^S$. Condition (ii) could be phrased as: $f$ is in the closure of $\P^S(\C^n)$ in the topology on $\mathcal C(\widehat{K}^S)$ generated by the seminorms $\|\cdot\|_A$ for all compact $A\subset \widehat{K}^S$.

\medskip

Any neighborhood $X$ of $\widehat{K}^S$ for a Reinhardt set $K$ contains a Reinhardt neighborhood of $\widehat{K}^S$, which contains the origin. Any function $f\in \O(X)$ can therefore be  expressed uniquely as a power series centered at the origin which is normally convergent in a neighborhood of $\widehat{K}^S$. 
If $K$ also contains a point $z\in\C^{*n}$, it contains the polycircle $\set{(|z_1|e^{i\theta_1},\dots, |z_n|e^{i\theta_n})}{\theta\in \R^n}$. In that case a holomorphic function $f$ is approximable by $\P^S(\C^n)$ precisely if its power series is of the form 
\begin{equation}\label{f_series}
    f(z)=\sum_{\alpha\in \R_+S}a_\alpha z^\alpha.
\end{equation}
If $K$ contains no point of $\C^{*n}$, then $K$ is a subset of $\C^n\setminus \C^{*n}$ which is pluripolar so 
we do not expect to determine the series of $f$ only by its values on $K$. Instead we can say that $f$ coincides on $K$ with some function whose series is of the form \eqref{f_series}. In  Proposition \ref{extend_holo} we see that if $f$ has a series expansion of the form \eqref{f_series} only known to be convergent in a neighborhood of a compact $K\subset \C^{*n}$, then $f$ extends as a holomorphic function in a neighborhood of $\widehat{K}^S$. The implication (v)$\Rightarrow$(i) is then true for the extension of $f$. 

\medskip

Theorem \ref{bigTheorem} is proven using Hörmander's $L^2$-theory. To construct the weights used in the proof, we use extremal plurisubharmonic functions studied in a series of papers \cite{ MagSigSig:2023, MagSigSigSno:2023, Sno:2024}, which we shall recount here. The \emph{supporting function} $\varphi_S\colon \R^n\to \R$ of $S$ is defined by $\varphi_S(\xi)=\sup_{s\in S}\scalar{s}{\xi}$, where $\scalar\cdot\cdot$ is the usual inner product on $\R^n$.  The \emph{logarithmic supporting function} $H_S$ of $S$ is defined by
\begin{align*}
    H_S(z)= \varphi_S({\Log\, z}), \qquad &z\in \C^{*n}\\
    H_S(z)=\varlimsup_{\C^{*n}\ni w\to z} H_S(w), \qquad &z\in \C^n\setminus\C^{*n}.
\end{align*}
where $\Log\, z= (\log|z_1|, \dots , \log|z_n|)$.
Let $\L^S(\C^n)$ be the class of all $u\in \PSH(\C^n)$ such that $u\leq H_S+c_u$ for some constant $c_u$  only depending on $u$.
Theorem 3.6 of \cite{MagSigSigSno:2023} implies that $p\in \O(\C^n)$ is a member of $\P^S_m(\C^n)$ if and only if $\log (|p|^{1/m})$ belongs to $ \L^S(\C^n)$. 

\medskip

 We associate to  $S$ and a compact $K\subset \C^n$  the $m$-th \emph{Siciak functions} 
\begin{align*}
    \Phi^S_{K,m}(z)= \sup\set{|p(z)|^{1/m}}{p\in \P^S_m(\C^n),\, \|p\|_K\leq 1},
\end{align*}
the \emph{Siciak function}, which can by \cite{MagSigSigSno:2023}, Proposition 2.2, be equivalently defined as 
$$\Phi^S_K=\sup_{m\in \N}\Phi^S_{K,m}=\lim_{m\to \infty}\Phi^S_{K,m},$$
and the \emph{Siciak-Zakharyuta function}
$$V^S_K(z)= \sup\set{u(z)}{u\in\L^S(\C^n),\, u|_K\leq 0}.$$
From the definition follows that $\widehat{K}^S= \set{z\in \C^n}{\Phi^S_K(z)=1}$. Replacing $S$ with $\overline{S\cap\Q^n}$ does not change the polynomial ring $\P^S(\C^n)$, so we may assume throughout the article that $S$ satisfies $\overline{S\cap\Q^n}=S$. Then \cite{MagSigSig:2023}, Theorem 1.1 states that $V^S_K(z)=\log\Phi^S_K(z)$ for all $z\in \C^{*n}$. Since $\widehat{K}^S\cap\C^{*n}= \set{z\in \C^{*n}}{V^S_K(z)=0}$, it is useful to derive a formula for $V^S_K(z)$.  

\medskip

\begin{proposition}\label{V^S_KformulaReinhardt} 
    Let $K\subset \C^{n}$ be a compact Reinhardt set and $A= \Log\, (K\cap\C^{*n})$. Assume  $S\cap \R^{*n}_+\neq \emptyset$ or  $K=\overline{K\cap \C^{*n}}$. 
    Then for all $z\in \C^{*n}$ we have
    \begin{equation}\label{Thm_1.3}
        V^S_K(z)=\sup_{s\in S} (\scalar s {\Log\, z} - \varphi_A(s)).
    \end{equation}
\end{proposition}

\medskip

This extends Proposition 4.3 in \cite{MagSigSigSno:2023}, that $V^S_{\T^n}=H_S$  where $\T$ is the unit torus. Proposition \ref{V^S_KformulaReinhardt} implies that $\Log$  maps $\widehat K^S\cap \C^{*n}$  onto the hull of $\Log\, K$ with respect to the cone $\Gamma=\R_+S$ in $\R^n$.
Furthermore, 
Proposition \ref{piJKJ} provides a description of $\widehat K^S$ on  $\C^n\setminus \C^{*n}$, giving us a complete description of $S$-hulls of Reinhardt sets in $\C^{*n}$. 

\medskip

We prove Proposition \ref{V^S_KformulaReinhardt} in Section \ref{sec:Shulls_of_Reinhardt_sets}.
 In Section 2, we explore the properties of $S$-hulls in terms of the properties of the set $S$ itself. We provide examples of compact sets that satisfy the hypothesis of Theorem \ref{bigTheorem}, by showing in Proposition \ref{compact_hull} that if $S$ has non-empty interior, then compact subsets $K$ of $\R^{*n}_+$ have $\widehat K^S\subset \R^{*n}_+$. Section \ref{sec:Circled} 
is devoted to the proof of Theorem \ref{approx_on_E} and Section 5 to the proof of Theorem \ref{bigTheorem}. 

\medskip

\subsection*{Acknowledgment}  
The results of this paper are a part of a research project, 
{\it Holomorphic Approximations and Pluripotential Theory},
with  project grant 
no.~207236-051 supported by the Icelandic Research Fund. The author was also funded by the Assistant Teacher grant at the University of Iceland. I would like to thank the Fund for its support and the 
Mathematics Division, Science Institute, University of Iceland,
for hosting the project. I thank my
supervisors Benedikt Steinar Magnússon and Ragnar Sigurðsson for their guidance and Bergur Snorrason for a careful reading of the manuscript. Upon completeon of this paper, the author is a postdoctoral fellow at the Institute of Mathematics, Physics and Mechanics in Ljubljana, Slovenia.

\section[$S$-convex sets]{ $S$-convex sets}\label{sec:S-convex}
In order to appreciate Theorem \ref{bigTheorem} it is necessary to identify some compact $S$-convex subsets of $\C^{*n}$, to know that it is not a statement on an empty family of sets. In fact, if $\widehat K^S$ is a compact subset of $\C^{*n}$, then Theorem \ref{bigTheorem} implies that $\widehat K^S=\widehat K$. Indeed, if $z\in \widehat{K}^S\setminus \widehat K$, there exists a polynomial $f$ such that $|f(z)|>\|f\|_K$. By Theorem \ref{bigTheorem}, there exists a $p\in \P^S(\C^n)$ such that $\|f-p\|_{\widehat K^S}<(|f(z)|- \|f\|_K)/2$. Then $|p(z)|> \|p\|_K$, which contradicts the assumption that $z\in \widehat{K}^S$. Therefore $\widehat{K}^S= \widehat K$. 

\medskip

Recall that $S\subset\R^{n}_+$ is a neighborhood of zero in $\R^n_+$ if there exists an $r>0$ such that $r\Sigma\subset S$, where  $\Sigma$ is the standard unit simplex. If $S$ is a neighbourhood of zero, then $\mathcal{P}^S(\C^n)$ contains all the polynomials in $\C^n$, so that case gives no new approximating result.  If $S$ is not a neighborhood of zero, there is some standard basis vector $e_k$ for $k\in[n]=\{1,\dots, n\}$ such that $e_k\notin \R_+S$. Every  $p\in\P^S(\C^n)$ can then be written of the form $p(z)= c+\sum_{j\neq k}z_jp'_j(z)$ where $c$ is a constant and $p'_j$ is a polynomial for $j\in[n]\setminus\{k\}$. Then $p$ is constant on the axis $\C e_k$ in $\C^n$. So if a set $K$ contains a point from $\C e_k$, then $\widehat{K}{}^S$ contains $\C e_k$, and can not be compact. 

\medskip

In the same vein we can identify more subspaces where polynomials from $\P^S(\C^n)$ are fixed depending on $S$, but for that we need some notation. For an ordered subset $J=(j_1,\dots, j_\ell)$ of $[n]$  we denote $\C^{J}=\set{z\in \C^n}{z_j = 0 \text{ if }j\notin J}$ and $\C^{*J}=\set{z\in \C^n}{z_j \neq 0 \text{ if and only if }j\in J}$ and define $\R^J=\C^J\cap\R^n$. 
 Let $\pi_J\colon\C^{n}\to \C^{\ell}$, $\pi_J(z)=(z_{j_1},\dots, z_{j_\ell})$ and let $S_J=\pi_J(S\cap \R^J)$. Proposition 3.3 of  \cite{MagSigSigSno:2023},  states that 
\begin{equation}\label{H_S(z)=H_S_J(pi_J(z)).}
     H_S(z)=H_{S_J}(\pi_J(z)), \qquad z\in \C^J.
 \end{equation}
 Every $p\in \P^S_m(\C^n)$ satisfies $\log|p|\leq {mH_S}+c$ for all $z\in \C^n$ and some constant $c$ by \cite{MagSigSigSno:2023}, Theorem 3.6. If $J\subset [n]$ is such that $S_J=\{0\}$, the right hand side of \eqref{H_S(z)=H_S_J(pi_J(z)).} is zero. Any $p\in \P^S(\C^n)$ is then bounded, hence constant, on $\C^J$, which is an unbounded set if $J\neq \emptyset$. 

 \medskip
 
 While singletons on $\C^n\setminus\C^{*n}$ may have unbounded $S$-hulls, we show in Corollary \ref{singleton} that the singletons in $\C^{*n}$ are $S$-convex if $S$ has non-empty interior, referred to as a \emph{convex body}. This follows from the fact that the family $\P^S(\C^n)$  {separates the points} of $\C^{*n}$, meaning that for every pair $x,y\in \C^{*n}$, $x\neq y$ there exists  $p\in \P^S(\C^n)$ such that $p(x)\neq p(y)$. Furthermore, this can be done only using monomials.

\begin{proposition}\label{monomials_separate_points}
    Let $S$ be a convex body. Then the monomials in $\P^S(\C^n)$ separate the points of $\C^{*n}$.
\end{proposition}

We use the notation $e^x=(e^{x_1},\dots, e^{x_n})\in \C^n$ for $x=(x_1,\dots, x_n)\in \C^n$.

\begin{proof}
    Let $z,w\in \C^{*n}$, $z\neq w$ and let $\xi,\eta,\theta,\phi\in \R^n$ be such that $e^{\xi_j+i\theta_j}=z_j$ and $e^{\eta_j+i\phi_j}=w_j$ for $j=1,\dots, n$. Assume first that $\xi\neq\eta $. Let the points $\alpha_1,\dots, \alpha_n\in \R_+S\cap\N^n$ be linearly independent. Then there is some $k\in [n]$ such that $\langle  \alpha_k,\xi\rangle \neq \langle  \alpha_k,\eta\rangle$ which implies that $|z^{\alpha_k}|\neq |w^{\alpha_k}|$. Otherwise, there exists some $j\in [n]$ such that $\theta_j-\phi_j\notin 2\pi\Z$.  Since the cone $\R_+S$ has nonempty interior, there exists some $\alpha\in \R_+S\cap \N^n$ such that $\alpha+\Sigma\subset \R_+S$. Then $\frac{1}{2\pi}\left(\scalar{\alpha+e_j}{\theta-\phi}-\scalar{\alpha}{\theta-\phi}\right)= \frac{1}{2\pi}\scalar{e_j}{\theta-\phi}$ is not an integer, so either $\beta=\alpha+e_j$ or $\beta=\alpha$ is such that $\frac{1}{2\pi}\scalar{\beta}{\theta-\phi}$ is not an integer. Then $z^\beta$ and $w^\beta$ have different arguments.
\end{proof}

\begin{corollary}\label{singleton}
    If $S$ is a convex body and $x\in \C^{*n}$, then  $\{x\}$ is $S$-convex.
\end{corollary}

We will show in Proposition \ref{compact_hull} that if $S$ is a convex body, then $S$-hulls of compact subsets of $\R^{*n}_+$ satisfy the hypothesis of Theorem \ref{bigTheorem}, that is are $S$-convex subsets of $\C^{*n}$. But first we need a Lemma.

\begin{lemma}\label{proper}
    Let $\alpha_1,\dots, \alpha_n\in \N^n$ be linearly independent. Then the map $F\colon \C^{*n}\to \C^{*n}$ defined by $F(z)=(z^{\alpha_1},\dots ,z^{\alpha_n})$ is proper.
\end{lemma}

\begin{proof}
    Let $L\colon \R^n\to \R^n$ be the linear map that maps $e_j$ to $\alpha_j$.
    Denote the adjoint of $L$ by $L^*$ and denote $\Log\,z=(\log|z_1|,\dots, \log|z_n|)$.
    Regard that $\Log\, F(z)=L^*\Log \,z$. 
    If $K$ is a compact set in $\C^{*n}$, it is contained in some polyannulus,
    $$K\subset\{z\in \C^n\,;\, e^{a_j}\leq |z_j|\leq e^{b_j}, \, j=1,\dots, n\}=\Log^{-1} \Big( \prod_{j=1}^n [a_j,b_j]\Big).$$
    Then 
    $F^{-1}(K)\subset (\Log\, F)^{-1} \big(\prod_{j=1}^n [a_j,b_j]\big)=(L^*\, \Log)^{-1} \big(\prod_{j=1}^n [a_j,b_j]\big).$

    Now $L^*$ is linear and bijective  since $L$ is, so $(L^*)^{-1} \prod_{j=1}^n [a_j,b_j] $ is compact, and is therefore contained in some box $\prod_{j=1}^n[c_j,d_j]$. Then finally 
    $$F^{-1}(K)\subset \Log^{-1}\Big(\prod_{j=1}^n[c_j,d_j]\Big)=\{z\in \C^n\,;\, e^{c_j}\leq |z_j|\leq e^{d_j}, \, j=1,\dots, n\},$$ so it is compact. We have shown that $F$ is proper as a map $\C^{*n}\to \C^{*n}$.\end{proof}

\medskip

\begin{proposition}\label{compact_hull}
Assume $S$ is a convex body 
and let $K$ be a compact subset of $\R^{*n}_+$. Then $\widehat K ^S$ is a compact subset of $\R^{*n}_+$.
\end{proposition}
\begin{proof}
    Let $\alpha_1,\dots, \alpha_n\in \R_+S\cap \N^n$ be linearly independent. By Lemma \ref{proper}, $F\colon \C^{*n}\to \C^{*n}$ defined by $F(z)=(z^{\alpha_1},\dots ,z^{\alpha_n})$ is proper. Now $F(K)$ is bounded and it is contained in some polyannulus $R\D^n\setminus (\mathbb H+r\D^n)$ where $\mathbb H= \C^n\setminus \C^{*n}$. Then $\widehat K^S$ is contained in $F^{-1}(R\D^n)$.  For each $j\in [n]$, regard the polynomial $p_j(z)=R-z^{\alpha_j}$. Then $\|p_j\|_K= R-|z^{\alpha_j}|\leq R$ and $|p_j(z)|>R-r$ whenever $|z^{\alpha_j}|<r$. The set $$\set{z\in \C^n}{|z^{\alpha_j}|<r, \text{ for some }j=1,\dots, n}= F^{-1}(\mathbb H+r\D^n)$$ does therefore not intersect $\widehat K^S$. Hence $\widehat K^S\subset F^{-1}(R\D^n\setminus (\mathbb H+r\D^n))\subset \C^{*n}$, which is compact because $F$ is proper. 

     Let $z\in \C^{*n}\setminus \R^n$. By Proposition \ref{monomials_separate_points} there is an $\alpha\in \R_+S\cap \N^n$ such that $z^\alpha\neq r^\alpha$ for $r=(|z_1|,\dots,|z_n|)$. Since $|z^\alpha|=|r^\alpha|$, the arguments of $z^\alpha$ and $r^\alpha$ must be different. Now $r^\alpha\in \R_+$, so $z^\alpha\notin \R_+$. Let $p'(w)=w^\alpha$, and $K'= p'(K)$. Then $K'$ is a compact subset of $\R$, so it is polynomially convex. Then there exists a polynomial $p\in \P(\C)$ such that $|p(z^\alpha)|>\|p\|_{K'}$, which implies that $p\circ p'\in \P^S(\C^n)$ is such $|p\circ p'(z)|>\|p\circ p'\|_{K}$. Therefore $z\notin\widehat{K}^S$, and we conclude that  $\widehat{K}^S\subset \R^n_+$.  
\end{proof}

\bigskip

 The case when $S$ has empty interior is thoroughly studied in \cite{MagSigSig:2023}, and we will review some of the results from there.   Then the convex set $S$ is of some lower dimension $\ell<n$ and there exists a linear map  $L\colon \R^\ell\to\R^n$ whose image contains $S$. We may assume that $\overline{S\cap\Q^n}=S$ throughout this article, since replacing $S$ with $\overline{S\cap\Q^n}$ does not affect the polynomial ring $\P^S(\C^n)$. This condition is also the sufficient and necessary for \cite{MagSigSig:2023}, Theorem {1.2}. It states that $L$ can be chosen so it maps the lattice points $\Z^\ell\subset\R^\ell$ onto the lattice points of its image, that is $L({\Z^\ell})= (\operatorname{span}_\R S)\cap \Z^n$. Furthermore, $L$ can be chosen so the compact convex set $T=L^{-1}(S)$ is a subset of  $\R^\ell_+$.

 \medskip

By Theorem {1.2} of  \cite{MagSigSig:2023}, the map $F_L\colon \C^{*n}\to \C^{*\ell}$, $F_L(z)=(z^{L(e_1)},$ $z^{L(e_\ell)})$
is such that every polynomial $p\in \P^S_m(\C^n)$ can be factored as $p=p'\circ F_L$ for some  $p'\in \P^T_m(\C^\ell)$. The fiber of $p$ through a point $z\in \C^{*n}$ therefore contains the fiber of $F_L$ through $z$, which is an unbounded $(n-\ell)$-dimensional submanifold. Therefore the $S$-hull of any $K$ with $ K\cap \C^{*n}\neq \emptyset$ is unbounded. We can even parameterize $\widehat{\{z\}}{}^S$ for any point $z\in \C^{*n}$.  Let $\beta'_{1},\dots, \beta'_{n-\ell}\in \N^n$ be generators of the lattice orthogonal to $S$, that is $(^\perp\operatorname{span}_\R S)\cap \N^n$, and regard the dual set of vectors $\beta_{jk}=\beta'_{kj}$. By \cite{MagSigSig:2023}, Lemma {4.1}, 
    for every $z\in \C^{*n}$, the image of the
map $\Upsilon_z\colon \C^{(n-\ell)*}\to \C^{*n}$, $\Upsilon_z(t)=(z_1 t^{\beta_1},\dots, z_nt^{\beta_n})$ is  precisely the fiber of $F_L$ that contains $z$, all of which is contained in $\widehat{\{z\}}{}^S$. For any point $w\in \C^{*n}$ outside this fiber, that is satisfying $F_L(w)\neq F_L(z)$, there exists by Proposition \ref{monomials_separate_points} a monomial $p'\in \P^{T}(\C^\ell)$ such that $p'(F_L(w))\neq p'(F_L(z))$, which implies that $w\notin \widehat{\{z\}}{}^S$. We summarize:

 \begin{proposition}\label{Sempty_interior=>all_hulls_unbounded}
     If $S$ has empty interior, then for any point $z\in \C^{*n}$ we have $$\widehat{\{z\}}{}^S\cap \C^{*n}= F_L^{-1}(F_L(z))=\Upsilon_z[\C^{(n-\ell)*}].$$ Furthermore, the $S$-hull of any $K\subset \C^n$ that intersects $\C^{*n}$ is unbounded.
 \end{proposition}

\section[$S$-hulls of Reinhardt sets]{$S$-hulls of Reinhardt sets}\label{sec:Shulls_of_Reinhardt_sets}

The modulus of monomials is simply
$|z^\alpha|= e^{\scalar{\alpha}{\Log\,z}},$ $z\in \C^{*n}.$ In logarithmic coordinates, $x=\Log\,z$, the sublevel sets of monomials are halfspaces. We note that $$\widehat K^S\subset \set{z\in \C^n}{|z^\alpha|\leq \sup_{w\in K}|w^\alpha| \text{ for all }\alpha\in \R_+S\cap\N^n},$$
and letting $\Gamma=\R_+S$, we have
\begin{equation}\label{Log_hat_K^S_inclusion}
    \Log[\widehat K^S\cap \C^{*n}]\subset \set{x\in \R^n}{\scalar x \alpha  \leq 
    \varphi_A(\alpha) \text{ for all }\alpha\in \R_+S\cap\N^n}.
\end{equation}
If $\overline{S\cap\Q^n}= S$, as we are assuming throughout this paper,
the right hand side can be taken with $\alpha$ from $\R_+S\cap \Q^n$, which is dense in the cone $\Gamma=\R_+S$ . Then the right hand side equals
\begin{equation}\label{hatGamma}
    \widehat{A}_\Gamma= \set{x\in \R^{n}}{\langle x, \xi\rangle \leq \varphi_A(\xi) \text{ for all } \xi \in \Gamma}.
\end{equation}
where $\varphi_A(\xi)=\sup_{a\in A} \scalar a \xi$. Note that this definition of $ \widehat{A}_\Gamma$ differs from  \cite{MagSigSigSno:2023}, Definition 5.5
which is $\widehat{A}_\Gamma\cap \R^n_+$. The inclusion \eqref{Log_hat_K^S_inclusion} turns out to be an equality for a certain class of sets.

\begin{definition}
    We say that $K\subset \C^n$ is a \emph{Reinhardt set} if for all $z\in K$ and $\zeta\in \T^n$ then $\zeta z=(\zeta_1z_1,\dots, \zeta_nz_n)\in K$.  A \emph{Reinhardt domain} is a domain that is a Reinhardt set and contains the origin.
\end{definition}

As mentioned in the introduction, $\widehat{K}^S\cap \C^{*n}=\set{z\in \C^{*n}}{V^S_K(z)=0}$, so identifying  $V^S_K$ is helpful to find the $S$-hull of a set $K$. We will obtain a complete characterization of $\widehat K^S$ for Reinhardt sets from  Propositions \ref{conehull} and \ref{piJKJ}. Reinhardt sets $K$ have the 
property that if $u\in \PSH(\C^n)$ has $u|_K\leq 0$, then 
\begin{equation*}
    u'(z)= \sup_{\zeta\in \T^n}u(\zeta_1z_1,\dots, \zeta_n z_n)
\end{equation*}
is plurisubharmonic, rotationally symmetric in each coordinate, satisfies $u'|_K\leq 0$ and $u'\geq u$. This implies that $V^S_K$ can be taken as the supremum over the rotationally symmetric $u\in \L^S(\C^n)$ with $u|_K\leq 0$. 

\medskip

If $u\in \PSH(\C^n)$ is rotationally symmetric, then $v\colon \R^n\to \R$, $v(\xi)=u(e^\xi)$ is a convex function that is increasing in each variable. The convexity of $v$ can best be seen by taking a sequence $u_\ell\searrow u$ of $u_\ell\in \mathcal{C}^2\cap\PSH(\C^n)$. By replacing $u_\ell$ by $\sup_{\zeta\in \T^n}u_\ell( \zeta z)$, we may assume $u_\ell$ are also rotationally symmetric. Taking $v_\ell(\xi)=u_\ell(e^\xi)$ we observe that 
\begin{equation*}\label{eq:Second_derivatives_u_delta_v_delta}
    \frac{\partial^2 u_\ell(z)}{\partial z_j \partial \overline{z}_k}=\frac{1}{4z_j\overline{z}_k}\cdot\frac{\partial^2 v_\ell(\xi)}{\partial \xi_j \partial \xi_k}, \qquad z\in \C^{*n},
\end{equation*}
which implies that
\begin{equation}\label{eq:detHesse}
    \det\left(\frac{\partial^2 u_\ell(z)}{\partial z_j \partial \overline{z}_k}\right)= \frac{1}{4^n |z_1\cdots z_n|^2}\det \left(\frac{\partial^2 v_\ell(\xi)}{\partial \xi_j \partial \xi_k}\right)\bigg|_{\xi=\Log\,z\, ,}
\end{equation}
and the positivity of one side of this equation implies the positivity of the other. Therefore $u_\ell\in \PSH(\C^n)$ implies that $v_\ell$ is convex. The limit $v$ is therefore also convex. Let $\operatorname{Conv}(\R^n)$ denote the set of convex functions on $\R^n$ that are increasing in each variable. We have described a bijection between rotationally symmetric $\PSH(\C^n)$ functions and $\operatorname{Conv}(\R^n)$ since $u\in \PSH(\C^n)$ can be recovered by $u(z)=v(\Log\, z)$. If  $u\in \L^S(\C^n)$, then $v\leq \varphi_S + \text{\emph{constant}}$, which can be denoted by $v\preccurlyeq \varphi_S$. If $u\leq 0$ on $K$ then $v\leq 0$ on $A=\Log\, K$, which can be rephrased by $v\leq \chi_A$, where $\chi_A(\xi)=1$ if $\xi\in A$ and $\chi_A(\xi)=+\infty$ otherwise. If $\overline{K\cap\C^{*n}}=K$, then $v\leq \chi_A$ also implies that $u|_K\leq 0$, so in that case we have:
\begin{equation}\label{VSK(exi)}
    V^S_K(e^\xi)=\sup \{ v(\xi)\,;\, v\in \operatorname{Conv}(\R^n),\, v\leq \chi_{A} \text{ and }v\preccurlyeq \varphi_S \}.
\end{equation}

We will study the right hand side through the  {Legendre-Fenchel transform}, also known as the convex conjugate. The \emph{Legendre-Fenchel transform} 
of a function $\mu\colon \R^n\to\overline{\R}=[-\infty,+\infty]$ is the function 
\begin{align*}
    \Lg(\mu)&\colon \R^n\to \overline{\R}\\[5pt]
    \Lg(\mu)(\xi)&= \sup_{x\in \R^n}\scalar{\xi}{x}-\mu(x).
\end{align*}
It is easy to see that $\Lg(\chi_E)=\varphi_E$ and  $\Lg(\varphi_E)= \chi_{\ch E}$ where $\ch(E)$ denotes the closed convex hull of $E$. 
To handle unbounded inputs, we extend the summation on $\R$ to a symmetric operation $\dotplus\colon \overline{\R}\times \overline{\R}\to \overline{\R}$ by 
$$ (+\infty)\dotplus (-\infty)=+\infty, \qquad  (+\infty)\dotplus a=+\infty,  \qquad (-\infty)\dotplus a=-\infty,  \qquad \text{ for all }a\in \R.$$
It is well known that for any $\mu\colon \R^n\to \overline{\R}$, then $\Lg^2(\mu):= \Lg(\Lg(\mu))$ is the largest lower-semicontinuous convex function that is $\leq \mu$, so in that spirit we obtain:

\bigskip

\textbf{Proof of Proposition \ref{V^S_KformulaReinhardt}:}
First of all, $$\Lg(\Lg(\chi_A)\dotplus\chi_S)= \sup_{s\in S} \big(\scalar{s}{\cdot}-\varphi_A(s)\big) 
$$ is smaller than the right hand side of \eqref{VSK(exi)}. Let now $v\in \Conv(\R^n)$ have $v\leq \chi_A$ and $v\leq \varphi_S+c$. Now $\Lg$ is order-reversing so 
    \begin{align*}
        \Lg(v)\geq \Lg(\varphi_S+c)=\chi_S-c,
    \end{align*}
    and therefore $\Lg(v)\dotplus\chi_S=\Lg(v)$. Now $\Lg^2=\Lg\circ\Lg$ preserves (lower semicontinuous) convex functions by \cite{Hormander:convexity}, Theorem 2.2.4,
    so 

    $$v= \Lg^2(v)= \Lg(\Lg(v)\dotplus\chi_S)\leq \Lg(\Lg(\chi_A)\dotplus\chi_S).$$
    which proves the theorem in the case when $K=\overline{K\cap \C^{*n}}$. The other case will be finished by Lemma \ref{V^S_K'|_K'=0}. \hfill $\square$

\begin{lemma}\label{V^S_K'|_K'=0}
Assume that $S\cap \R^{*n}_+\neq \emptyset$.  Let $K\subset\C^n$ be a compact Reinhardt set, and let $K'=\overline{K\cap \C^{*n}}$. Then $V^S_K|_{\C^{*n}}=V^S_{K'}|_{\C^{*n}}$ and $V^{S*}_{K'}|_{K'}=0$.
\end{lemma}

\begin{proof} 
    Clearly $V^S_K\leq V^S_{K'}$. Let $s\in S\cap \R^{*n}_+$. The function $v(z)= \scalar{s}{\Log\,z}\in \L^S(\C^n)$ takes the value $-\infty$ on $\C^n\setminus \C^{*n}$ and but values in $\R$ on $\C^{*n}$. 
    
    Let $u\in \L^S(\C^n)$ be such that $u|_{K'}\leq 0$ and let $\varepsilon>0$. 
    Then $(1-\eta)u+\eta v\in \L^S(\C^n)$  is $\leq \varepsilon$ on $K$ for all $\eta>0$ small enough. Therefore $u\leq V^S_K+\varepsilon$ on $\C^{*n}$, and since that is true for all $\varepsilon>0$, we have $u\leq V^S_K$ on $\C^{*n}$. We conclude that $V^S_{K'}=V^S_K$ on $\C^{*n}$. 

     If $u\in \PSH(\C^n)$ is rotationally symmetric and $u(z)\leq 0$ for some point $z\in \C^{*n}$, then $u\leq 0$ on the polydisc $D_z=\set{w\in \C^n}{|w_j|\leq|z_j|,\,j=1,\dots, n}$. This implies that for $\widetilde K=\bigcup_{z\in K\cap\C^{*n}} D_z$, we have  $V^S_{K'} \leq V^S_{\widetilde K}$ on $\C^n$. Furthermore, $V^{S}_{K}|_{\C^{*n}}=  V^{S}_{K'}|_{\C^{*n}}\geq V^{S}_{\widetilde K}|_{\C^{*n}}$, concluding the proof of the first statement.

Each $w\in \widetilde K$ is a member of $D_z$ for some $z\in \C^{*n}$, and $D_z$ is the unit ball in some norm. Lemma 5.2
in \cite{MagSigSigSno:2023} therefore implies that $\widetilde K$ is locally $\L$-regular. By \cite{MagSigSigSno:2023}, Proposition 5.3 and 5.4
$V^{S*}_{\widetilde K}|_{\widetilde K}=0$, from which follows that $V^{S*}_{K'}|_{K'}=0$.
\textcolor{white}{.}
\end{proof}

\medskip

The identity $\widehat{K}^S\cap \C^{*n}=\set{z\in \C^{*n}}{V^S_K(z)=0}$ now provides a corollary:
\begin{corollary}\label{conehull}
    Assume that $S\cap \Q^{*n}_+\neq \emptyset$ or $K=\overline{K\cap \C^{*n}}$.   Let $K$ be a compact Reinhardt subset of $\C^{n}$, $A=\Log(K\cap \C^{*n})$ and $\Gamma=\R_+S$.
    Then  $$\widehat{K}^S\cap \C^{*n}= \Log^{-1}\widehat{A}_\Gamma.$$ 
    
\end{corollary}

\medskip

Now $\widehat{A}_\Gamma$ can be described in terms of the dual cone $$\Gamma^\circ= \set{x\in \R^n}{\scalar x\xi \geq 0 \quad\forall \xi \in \Gamma}.$$
The dual cone of $\Gamma=\R_+S$ is $\Gamma^\circ=-\mathcal{N}(\varphi_S)=\set{\xi\in \R^n}{\varphi_S(-\xi)=0}$. 
The following is a slight modification of \cite{MagSigSigSno:2023}, Proposition 5.6.

\begin{proposition}\label{-dualcone}
Let $A$ be a subset of $\R^n$ with $0\in S$ and
$\Gamma$ be a proper closed convex  cone.  Then 
\begin{equation*}
\widehat A_\Gamma=\operatorname{ch}A-\Gamma^\circ.
\end{equation*} 
\end{proposition}

\begin{proof}  Take 
$a\in \operatorname{ch}A$ and $t\in \Gamma^\circ$ and let $x=a-t$.   For every $\xi\in \Gamma$
we have $\scalar t\xi\geq 0$ which implies 
$\scalar x\xi=\scalar a\xi-\scalar t\xi\leq \varphi_A(\xi)$
and $a\in \widehat A_\Gamma$.  

Conversely, we take  $x\not\in \operatorname{ch}A-\Gamma^\circ$ and prove
that $x\not \in \widehat A_\Gamma$.    
Since $\operatorname{ch}A-\Gamma^\circ$ is convex
the Hahn-Banach theorem implies that
$\{x\}$ and $\operatorname{ch}A-\Gamma^\circ$ can be separated
by an affine hyperplane.   Hence 
there exist $\xi\in \R^n$ and $c\in \R$ such
that $\scalar x\xi>c$ and $\scalar a\xi\leq c$ for every $a\in \operatorname{ch}A-\Gamma^\circ$.
By replacing $c$ with $\sup_{a\in \operatorname{ch}A-\Gamma^\circ}\scalar a\xi$
we may assume
there exists $a\in \operatorname{ch}A$ and $t\in \Gamma^\circ$ with
$\scalar {a-t}\xi=c$. Now we need to prove that
$\xi\in \Gamma=\Gamma^{\circ\circ}$ by showing that  $\scalar y\xi\geq
0$ for every $y\in \Gamma^\circ$.    Since $\Gamma^\circ$ is a convex
cone, we have $t+y\in \Gamma^\circ$ and 
$c-\scalar y\xi=\scalar{a-t-y}\xi \leq c$. Hence
$\scalar y\xi\geq 0$.
This implies that  $\scalar x\xi>c\geq \varphi_{S}(\xi)$ and we
conclude that  $a\not\in \widehat A_\Gamma$.  
\end{proof}

\bigskip

Several times in the next section will we use the fact that if $K$ is a polycircle with polyradius $ e^\varrho\in \R^{*n}_+$, then $\widehat{K}^S\cap\C^{*n}=  \Log^{-1}\widehat{\{\varrho\}}_\Gamma  = \Log^{-1}(\varrho +\mathcal{N}(\varphi_S))$, by Proposition \ref{-dualcone}. We can also describe $\widehat K^S$ for a Reinhardt $K$ on $\C^n\setminus\C^{*n}$,  though the description will not be as crisp as in Corollary \ref{conehull}. 
 The description will be in terms of the non-zero coordinates $J\subset [n]$ of the points. For that we use notation laid out in Section 2, and let $K_J=\pi_J(K)$.

\begin{proposition}\label{piJKJ}
    Let $K\subset \C^n$ be a compact Reinhardt set. For every ordered subset $J\subset\{1,\dots, n\}$ and every $z\in \C^J$ we have  $V^S_K(z)= V^{S_J}_{K_J}(\pi_J(z))$ and  $\Phi^S_K(z)= \Phi^{S_J}_{K_J}(\pi_J(z))$. Furthermore, $\pi_J(\widehat{K}^S\cap \C^{J})= {\widehat{K_J}}{}^{S_J}$. 
\end{proposition}

\begin{proof}
    By rearranging the coordinates we may assume that $J=\{1,\dots, \ell\}$.  Regard a point $z\in \C^J$, so $z= (z',0)$ with $z'\in \C^\ell$. If $u'\in\L^{S_J}(\C^\ell)$ has $u'|_{K_J}\leq 0$, then $u(z',z'')=u'(z')$ defines member of $\L^S(\C^n)$ with $u|_K\leq 0$. Then $u'(z')=u(z',0)\leq V^S_K(z',0)$, implying that $V^{S_J}_{K_J}(z')\leq V^S_K(z',0)$.

    For $u\in \L^S(\C^n)$ with $u|_K\leq 0$, we regard $u'(z')=u(z',0)$. Proposition 3.3
    from \cite{MagSigSigSno:2023} then implies that $u'(z')\leq H_S(z',0)+c_{u}= H_{S_J}(z')+c_{u}$, so $u'\in \L^{S_J}(\C^\ell)$.  Since $K$ is Reinhardt,  $K_J\times\{0\}^{n-\ell}\subset \widehat K$, which implies that $u|_{K_J\times\{0\}^{n-\ell}}\leq 0$.  Then $u'|_{K_J}\leq 0$ and $u(z',0)=u'(z')\leq V^{S_J}_{K_J}(z')$. Therefore $V^S_K(z',0)\leq V^{S_J}_{K_J}(z')$.

    Repeat this argument except let $u=\log|p|^{1/m}$ with $p\in \P^S_m(\C^n)$  and let $u'=\log|p'|^{1/m}$ with $p'\in \P^{S_J}_m(\C^\ell)$ for some $m\in \N$. That  yields $\log \Phi ^S_K(z',0)= \log \Phi^{S_J}_{K_J}(z')$.     The last assertion then follows from the fact that $\widehat K^S= \set{z\in \C^n}{\Phi^S_K(z)=1}$.
\end{proof}
\bigskip

If $S\cap \R^{*n}_+= \emptyset$, then there is some largest $J\subset [n]$ of $\#J=\ell<n$ such that $S_J\cap\R^{*\ell}_+\neq \emptyset$. Then $V^S_K$ is independent of its variables from $[n]\setminus J$ and can be described by $V^S_K(z)=V^{S_J}_{K_J}(\pi_J(z))$, and the right hand side can be understood by Theorem \ref{V^S_KformulaReinhardt}.

\medskip

We can now determine the $S$-hull of the unit polydisc $\T^n$. By Corollary \ref{conehull}, we have $(\widehat{\mathbb T^n})^S\cap\C^{*n}=\set{z\in \C^{*n}}{H_S(z)=0}$. Proposition 3.3
from \cite{MagSigSigSno:2023} shows that for $a\in \C^J$ we have $H_S(a)=H_{S_J}(\pi_J(a))$. By Proposition \ref{piJKJ}, then $(\widehat{\mathbb T^n})^S=\set{z\in \C^{n}}{H_S(z)=0}$. Furthermore, this is an unbounded set if and only if $S$ is not a neighborhood of zero. Another implication of Proposition \ref{piJKJ} is that if $K$ is Reinhardt, then  $V^S_K=\log\Phi^S_K$ is true everywhere in $\C^n$.

\section[Approximation on Reinhardt sets]{Approximation on Reinhardt sets}\label{sec:Circled}

Theorem \ref{approx_on_E} will be proven over the course of a few propositions. The step (ii)$\Rightarrow$(iii) was motivated in the introduction, and the steps (iii)$\Rightarrow$(iv) and (i)$\Rightarrow$(ii) are clear. Next we prove (iv)$\Rightarrow$(v) with the simplification that $K$ contains a point from $\C^{*n}$. 
Then $K$ contains some polycircle in $\C^{*n}$ centered at 0, 
which implies that the values of $f$ on $K$ determine the function completely, which simplifies condition (v). This extra condition will be removed in Proposition \ref{circ_a_alpha=0_ekkiC^*n}. 
Holomorphic function defined by a convergent power series of the form $f(z)= \sum_{\alpha\in \Gamma} c_\alpha z^\alpha$, where $\Gamma=\R_+S$, have partial sums from $\P^S(\C^n)$ and they provide a uniform approximation of the holomorphic function on  compact subsets of its domain of convergence.

\medskip

The $S$-hull of a non-empty Reinhardt set $K$ is always a connected Reinhardt set that contains zero, 
so if $X$ is a neighborhood of $\widehat K^S$, then $X$ 
contains a Reinhardt domain that contains $\widehat K^S$, and any holomorphic function on such a domain is uniquely expressed as a convergent power series. 

\begin{lemma}\label{circ_a_alpha=0}
    Let $X$ be a neighborhood of $\widehat{K}^S$ where $K$ is a compact Reinhardt set with $K\cap \C^{*n}\neq \emptyset$. Let $f\in \O(X)$ have series expansion $f(z)= \sum_{\alpha\in \N^n} a_\alpha z^\alpha$ around zero. If $f$ is bounded on $\widehat{K}^S$, then $a_\alpha=0$ for all $\alpha\notin \R_+S$.
\end{lemma}

\begin{proof} The dual of 
 $\Gamma=\R_+S$ is  $\Gamma^\circ=-\mathcal{N}(\varphi_S)= \set{\xi \in \R^n}{\varphi_S(-\xi)=0}$. Since $\Gamma$ is a closed convex cone then $\Gamma^{\circ\circ}=\Gamma$. So if $\alpha\notin \R_+S$, there exists a $\xi'\in \Gamma^\circ$ such that $\scalar \alpha{\xi'}< 0$. Then $\xi=-\xi'$ has $\varphi_S(\xi)=0$ and $\scalar \alpha \xi>0$.

Let $z\in K\cap \C^{*n}$ and $r=(|z_1|,\dots, |z_n|)$. Let $C_t$ denote the polycircle with center $0$ and
polyradius $(r_1e^{t\xi_1},\dots,r_ne^{t\xi_n})$. Corollary \ref{conehull} implies that $C_t\subset E$ for all $t\geq 0$. The component $\Omega'$ of $\cap_{\zeta\in \T^n} \zeta\Omega$ that contains $\widehat{K}^S$ is a Reinhardt domain that contains zero. By \cite{Hormander:SCV}, Theorem 2.4.5,  $f$ is expressible by a normally convergent power series $f(z)=\sum_{\alpha\in\N^n}a_\alpha z^\alpha$ in $\Omega'$. By
the Cauchy formula for derivatives we have
\begin{equation*}
a_\alpha=
\dfrac 1{(2\pi i)^n} \int_{C_t}
\dfrac{f(\zeta)}{\zeta^\alpha} \, 
\dfrac{d\zeta_1\cdots d\zeta_n}{\zeta_1\cdots \zeta_n}. 
\end{equation*}
For $\zeta=(r_1e^{t\xi_1+i\theta_1},\dots,r_ne^{t\xi_n+i\theta_n})$ 
on $C_t$ we have $|f(\zeta)|/|\zeta^\alpha|\leq
\|f\|_E\cdot r^{-\alpha}e^{-t\langle \alpha,\xi\rangle}$. The right hand side
tends to $0$ as $t\to +\infty$ and we conclude that  $a_\alpha=0$.
\end{proof}
\bigskip

The assumption $K\cap \C^{*n}\neq \emptyset$ cannot be removed from the previous result. Take for example $S= \operatorname{ch}\{(0,0), (1,0), (1,1)\}\subset \R^2_+$ and $K\subset \{0\}\times \C$. Then $\widehat{K}^S= \{0\}\times \C$ and $z^{(1,2)}$ is bounded on $\widehat{K}^S$, even though $(1,2)\notin S$. What we can say instead is that $f$ coincides on $\widehat{K}^S$ with a function that satisfies the conclusion of Proposition \ref{circ_a_alpha=0}.

\begin{proposition}\label{circ_a_alpha=0_ekkiC^*n}
    Let $X$ be a neighborhood of $\widehat{K}^S$ where $K$ is a compact Reinhardt set. Let $f\in \O(X)$ be bounded on $\widehat{K}^S$. Then there exists a $h\in \O(X)$ with a convergent series expansion  $h(z)= \sum_{\alpha\in \R_+S} a_\alpha z^\alpha$ around zero such that $f=h$ on $\widehat{K}^S$.
\end{proposition}

\begin{proof} In this proof we will use the index $k\in [n]= \{1,\dots,n\}$ to denote the ordered subset $J=[n]\setminus\{k\}$ in the notation from Section 2. To be precise, denote $Z_k = \set{z\in \C^n}{z_k=0}$ and  $X_k=\set{x\in \R^n}{x_k=0}.$
Let $\pi_k\colon \C^n\to \C^{n-1}$ denote $\pi_k(z)=(z_1,\dots, z_{k-1}, z_{k+1},\dots,z_n)$ and use the same notation for its restriction to $\R^{n-1}$. Let $K_k= K\cap  Z_k$ and $S_k= \pi_k(S\cap X_k)\subset \R^{n-1}_+$. Additionally, we define $\mu_k\colon \C^{n-1}\to \C^n$ by $\mu_k(w)=(w_1,\dots, w_{k-1}, 0, w_{k},\dots, w_{n-1})$. Observe that for all $k\in [n]$ we have $\pi_k\circ \mu_k=\operatorname{id}_{\C^{n-1}}$ and on $Z_k$ we have $\mu_k\circ \pi_k=\operatorname{id}_{Z_k}$.

    We induct over the dimension $n$. The base case $n=1$ is easy: We covered the case $K\cap\C^*\neq\emptyset$ in Lemma \ref{circ_a_alpha=0} and if $K\cap\C^*=\emptyset$ then $K=\{0\}$. If $S=\{0\}$
    then $\widehat{K}^S=\C$ and $f$ is constant, and $h=f\in \P^S(\C)$. If $S=[0,s]$, $s>0$ then $\widehat{K}^S=\{0\}$. The constant function $h=f(0)$ is in $\P^S(\C)$ and $f=h$ on $\widehat K^S$.

    Assume the result  is true in the dimension $n-1$. If $\widehat{K}^S$ contains a point $z\in \C^{*n}$ the result follows by Lemma \ref{circ_a_alpha=0} for the set $K\cup\{z\}$. So we may assume that $\widehat{K}^S\subset \C^n\setminus \C^{*n}$. 
    The holomorphic function $f\circ\mu_k$ 
    is bounded in a neighborhood of the $\P^{S_k}(\C^{n-1})$-hull of  $ \pi_k(K_k)$ by Proposition \ref{piJKJ}. 
    
    By the inductive hypothesis, there exists a holomorphic $h_k$ with a series expansion  $h_k(w)=\sum_{\beta\in S_k}{a}^k_\beta w^\beta$  defined in a neighborhood $\Omega_k$ of $\pi_k(K_k)\subset \C^{n-1}$ such that $f\circ\mu_k=h_k$ on $\pi_k(K_k)$. Then $f $ coincides with $h_k\circ \pi_k$ on $K_k$.  

    If $z\in K_j\cap K_k= K\cap Z_j\cap Z_k$, then $f(z)=h_j(\pi_j(z))= h_k(\pi_k(z))$ which implies that 
    $\sum_{\alpha\in S\cap X_j\cap X_k}{c}^j_{\pi_j(\alpha)} z^\alpha= \sum_{\alpha\in S\cap X_j\cap X_k}{c}^k_{\pi_k(\alpha)} z^\alpha$. This implies that ${c}^k_{\pi_k(\alpha)}$ is the same number for all $k$ such that $\alpha_k=0$. We put ${a}_\alpha={c}^k_{\pi_k(\alpha)}$ for some $k$ such that $\alpha_k=0$ and ${a}_\alpha=0$ if $\alpha\in \R^{*n}$ and regard the series $h(z)= \sum_{\alpha\in \R_+S} {a}_\alpha z^\alpha$. Then $h= h_k\circ \pi_k$ on $K_k$,  therefore $f=h$ on $\widehat{K}^S$.
    
    We also need to show that $h$ is convergent. Since each of the series $h_k(w)$ is convergent in the neighborhood $\Omega_k$ then 
    by \cite{Hormander:SCV}, Theorem 2.4.2, there exists a constant $C_k>0$ such that $|c^k_{\beta} w^\beta|\leq C_k$ for all $\beta\in \N^{n-1}$. Take $C= \max_{k\in [n]} C_k$, and observe that $|a_\alpha z^\alpha|\leq C_k$ for all $\alpha\in \N^n$ on the neighborhood $\bigcap_{k=1}^n \pi_k^{-1}\Omega_k$ of $K$.
    \end{proof}
\bigskip

What remains to prove in Theorem \ref{approx_on_E} is the implication (v)$\Rightarrow$(i).

\begin{proposition}
    \label{(v)=>(i)}
    Let $X$ be a neighborhood $\widehat{K}^S$ for a compact Reinhardt set $K\subset \C^n$. 
    Assume that $f\in \O(X)$ has a series expansion centered at zero of the form $f(z)= \sum_{\alpha\in \R_+S} {a}_\alpha z^\alpha$.
    Then $f$ can be approximated uniformly on $\widehat{K}^S$ by polynomials from $\P^S(\C^n)$.
\end{proposition}

\begin{proof}
    The polynomials $f_N(z)= \sum_{|\alpha|\leq N} {a}_\alpha z^\alpha$ converge to $f$ uniformly on $K$.
    For each $N$ and each of the finitely many points $\alpha \in N\Sigma\cap \R^+S$ there is an $m$ such that $\alpha\in mS$. Therefore there is an $m_N\in \N$ such that  $f_N\in \P^S_{m_N}(\C^n)$. By rearranging the indices, we have a sequence  $f_m\in \P^S_{m}(\C^n)$ tending to $f$ uniformly on $K$. So let $\varepsilon>0$ and $m\in \N$ be such that $\Tilde f= f-f_m$ has $\|\Tilde f\|_{K}\leq \varepsilon$.

    For $z_0\in \widehat{K}^S$ we want  to show that $|\Tilde{f}(z_0)|\leq\varepsilon$ as that would imply that $f_m\to f$ uniformly on $\widehat{K}^S$. Recall that we may assume that $\overline{S\cap \Q^n}=S$. Rearrange the coordinates if needed so 
    \begin{align*}
        \{1, \dots, \ell\}&= \set{j\in [n]}{z_{0,j}\neq 0\text{ and }\exists s\in S, s_j\neq 0}\\
        \{\ell+1, \dots, \ell+k\}&= \set{j\in [n]}{z_{0,j}\neq 0\text{ and }\forall s\in S, s_j= 0}\\
         \{\ell+k+1, \dots, n\}&= \set{j\in [n]}{z_{0,j}= 0}
    \end{align*}
    Write $z_0=(z_0', z_0'',0)$ with $z_0'\in \C^{*\ell}$, $z_0''\in \C^{*k}$. Let $J=\{1,\dots, \ell\}$ and use the notation laid out in Section 2. Then $S_J\cap \R^{*\ell}_+\neq 0$.  The function $$u(z')= \max_{\zeta\in \T^\ell}\log|\Tilde{f}(\zeta z',0,0)|$$ is plurisubharmonic on the open Reinhardt set $\Omega=\cap_{\zeta\in \T^\ell} \zeta \cdot\pi_J(X\cap\C^J)\subset \C^\ell$ which is a neighborhood of $K_J$ and $u(\zeta z)= u(z)$ for all $\zeta\in \T^\ell$.

    Then  $\R^\ell\ni x\mapsto u(e^x)$ is a convex function of $\ell$ real variables.  Since $u(e^x)\leq \log \varepsilon$ holds for $x\in \Log\, K_J$, it is also true for $x\in \ch(\Log\, K_J)$. By Proposition \ref{piJKJ}, $z'\in \widehat K_J^{S_J}\subset \Omega$. By Proposition  \ref{conehull}, we have $$\widehat K_J^{S_J}\cap\C^{*\ell}= \Log^{-1}\big(\ch(\Log \, K_J)+\mathcal{N}(\varphi_{S_J})\big),$$ so $\Log\, z_0=\varrho +\xi$ for some $\varrho\in \ch(\Log\, K_J)$ and $\xi\in \mathcal{N}(\varphi_S)$. 
    Furthermore, $e^{\varrho + t\xi}\in \widehat{K}_J^{S_J}$ and therefore $(e^{\varrho + t\xi},0)\in  \widehat{K}^S$ for all $t\geq 0$. 
    
    The function $v\colon\R_+\to \R$, $v(t)=u(e^{\varrho+ t\xi})$ is convex and $v(0)\leq \log \varepsilon$. Since and both $f$ and $f_m$ are bounded on $\widehat{K}^S$, $v$ is bounded above on $\R_+$.  Bounded convex functions on $\R_+$ are decreasing, so we can conclude that $v\leq \log \varepsilon$. In particular $v(1)\leq \log \varepsilon$ so $|\Tilde{f}(z_0',0,0)|\leq \varepsilon$. Since $\Tilde{f}$ is independent of the $(\ell+1),\dots, (\ell+k)$-th variables, we can conclude that  $|\Tilde{f}(z_0)|\leq \varepsilon$.
\end{proof}
\bigskip

If $f$ is expressed by a power  series of the form \begin{equation}\label{series}
    f(z)= \sum_{\alpha\in \R_+S} {a}_\alpha z^\alpha
\end{equation}with uniform convergence on some compact set $K$, then the domain of convergence of \eqref{series} is some Reinhardt domain that contains $K$. The domains of convergence of power series are well characterized in Section 2.4 of Hör\-mander's book \cite{Hormander:SCV}. His Theorem 2.4.3 asserts that the domain of convergence $\Omega$ of any power series is such that its image in logarithmic coordinates $\Log\, \Omega$ is an open  convex set and that for every $\varrho\in \Log\, \Omega$, then $\varrho-\R^n_+\subset \Log\, \Omega$. By Proposition \ref{-dualcone}, this can be phrased so that the closure of $\Log\, \Omega$ is convex with respect to the cone $\R^n_+= \R_+\Sigma$. With the added information that the series in question is of the form \eqref{series}, we can show that the closure of its domain of convergence in logarithmic coordinates is convex with respect to the cone $\R_+S$.

\begin{proposition}\label{extend_holo}
    Let $\Omega$ be a Reinhardt domain containing the origin and let $f\in \O(\Omega)$ have series expansion \eqref{series}. Denote $D= \Log(\Omega\cap\C^{*n})$ and let $\Gamma= \R_+S$. Let $\widetilde\Omega$ be the interior of 
    $\Log^{-1} (\widehat{D}_\Gamma)$. Then the convergence of \eqref{series} is normal in $\widetilde \Omega$ and $f$ extends to a holomorphic function on $\widetilde \Omega$. 
\end{proposition}

\begin{proof}
    By \cite{Hormander:SCV}, Theorem 2.4.2, the series \eqref{series} is normally convergent in the interior of the set $B$ of all $z\in \C^n$ such that $|c_\alpha z^\alpha|\leq C$ for all $\alpha\in \N^n$. By \cite{Hormander:SCV}, Theorem 2.4.6,  $\Log (B\cap\C^{*n})$ contains $\ch\,D$.  If $z\in\Log^{-1}(\widehat D_{\Gamma})$, then $\Log\, z= b+\xi$ where $b\in \ch D$ and $\varphi_S(\xi)=0$  by Proposition \ref{-dualcone}. For any $\alpha\in \R_+S\cap \N^n$ we then have $|c_\alpha z^\alpha|= |c_\alpha|e^{\scalar \alpha{b+\xi}}\leq |c_\alpha| e^{\scalar \alpha{b}}\leq C$. This shows that $z\in B$ and that the open set $\widetilde\Omega$ lies in the interior of $B$ so \eqref{series} is normally convergent on $\widetilde \Omega$.
    \textcolor{white}{.} 
\end{proof}
\bigskip

We saw in \eqref{Log_hat_K^S_inclusion} that the $S$-hull of compact set $K$ restricted to $\C^{*n}$ is contained in $\Log^{-1}\widehat{A}_\Gamma$ where $A= \Log\, K$. Whenever $\Omega$ is a Reinhardt domain, then $\widetilde \Omega$ will contain the $S$-hull of every compact subset of $\Omega$.

\begin{corollary}
    Let $\Omega$ be a Reinhardt domain containing the origin and let $K$ be a compact subset of $\Omega$. If $f\in \O(\Omega)$ has a series expansion \eqref{series}, then $f$ extends as  a holomorphic function in a neighborhood of $\widehat{K}^S$.
\end{corollary}

Furthermore if $X$ is any neighborhood of a compact Reinhardt set $K$, then the connected component of  $\Omega=\cap_{\zeta\in \T^n} \zeta X$ that contains $K$ is a Reinhardt domain. Note that Proposition \ref{extend_holo} implies that if $f$ is only assumed to be defined on any connected Reinhardt domain containing $K$ and the origin, then $f$ can be extended to a neighborhood of $\widehat{K}^S$. By Proposition \ref{(v)=>(i)}, this extension can be approximated uniformly by $\P^S(\C^n)$ polynomials on $\widehat K^S$.

\section[Proof of Theorem \ref{bigTheorem}]{Proof of Theorem \ref{bigTheorem}}\label{sec:CompactHulls_Hörmander}

The main tool in the proof of Theorem \ref{bigTheorem} is
 Hörmander's $L^2$-methods, so 
for the reader's convenience we present his Theorem 4.2.6 from \cite{Hormander:convexity}.

\begin{theorem} {\bf(H\"ormander)} \label{thm:5.1} 
  Let $X$ be a pseudoconvex domain in  
$\C^n$ and $\varphi\in \PSH(X)$. Let $\varphi_a(z)=\varphi(z)+a\log(1+|z|^2)$ for some $a>0$.
For every $f\in L^2_{(0,1)}(X,\varphi_{a-2})$ 
satisfying $\bar\partial f=0$ 
there exists a solution 
$u\in L^2(X,\varphi_a)$ of 
$\bar\partial u=f$ satisfying the estimate
\begin{align}  
\label{eq:5.8}
\|u\|_{\varphi_a}^2&=\int_X |u|^2(1+|z|^2)^{-a} e^{-\varphi}\,
                     d\lambda \\
&\leq 
\dfrac 1a\int_X |f|^2(1+|z|^2)^{-a+2} e^{-\varphi}\, d\lambda
=\dfrac 1a \|f\|_{\varphi_{a-2}}^2.
\nonumber
\end{align}
If $f_j\in \mathcal{C}^\infty(X)$ for $j=1,\dots,n$, then $u\in \mathcal{C}^\infty(X)$. 
\end{theorem}

Our task is to find an appropriate weight $\varphi$ in order to be assured that the corresponding $L^2$-estimate implies that a holomorphic function is from $\P^S(\C^n)$. The prototype for such a result is Theorem 3.6
from \cite{MagSigSigSno:2023}, which states that an entire function $p$ on $\C^n$ is a member of $\P^S_m(\C^n)$, $m\in \N$ if and only if for some constants $C>0$ and $a$ smaller than the distance between $mS$ and $\N^n\setminus mS$ in the $L^1$-norm, we have
\begin{equation}\label{eq:p_Liouville}
   |p(z)|\leq C(1+|z|)^ae^{mH_S(z)}, \qquad z\in \C^n. 
\end{equation}

Therefore we need a result where a finite $L^2$-estimate can imply \eqref{eq:p_Liouville}. Our trick is that the weight should accommodate the real Jacobian of the change of coordinates into logarithmic coordinates.

\begin{proposition}\label{cor:u_uni_bound}
    Let $u\in \mathcal{C}^1(\C^n)$ be such that $\overline{\partial}u$ has compact support. Denote  $\nu(z)=\scalar{\mathbf{1}}{\Log\, z}$.  If
\begin{equation}
  \label{eq:6.10'}
\int_{\C^n} |u(z)|^2(1+|z|^2)^{-{a}} e^{-2mH_S(z)-2\nu(z)} d\lambda(z)
<+\infty,
\end{equation}
then there exists a constant $C>0$ such that
\begin{equation}\label{u_uni_bound}
    |u(z)| \leq C(1+|z|)^{a
    } e^{mH_S(z)}, \qquad z\in \C^n.
\end{equation}
\end{proposition}

\begin{proof}
The Jacobi determinant of 
$\zeta\mapsto e^\zeta=(e^{\zeta_1},\dots,e^{\zeta_n})$ viewed as a mapping
$\R^{2n}\to \R^{2n}$ is equal to 
$|e^{\zeta_1}\cdots e^{\zeta_n}|^2=e^{2\scalar{{\mathbf 1}}\xi}$.
We define $v\colon \C^n\to \C$ by
$v(\zeta)=u(e^{\zeta})$, write
$\zeta=\xi+i\eta$, for $\xi,\eta\in \R^n$. 
Let 
$$A=\set{\zeta= \xi+i\eta\in \C^n}{\xi\in \R^n\text{ and }\eta_j\in [-\pi,\pi],\, j=1,\dots, n}$$

By (\ref{eq:6.10'}) we get the estimate 
\begin{multline*}
\int_{A} |v(\zeta)|^2  (1+|e^{\zeta}|^2)^{-{a
}}
e^{-2m\varphi_S(\xi)} d\lambda(\zeta)\\
=  \int_{\C^n} |u(z)|^2(1+|z|^2)^{-{a
}}  
e^{-2mH_S(z)-2\nu(z)} d\lambda(z)
<+\infty.
\end{multline*}
Since  $\bar\partial u$ has compact support it follows that
there exists a constant $C_2>0$ such that 
\begin{equation}
|\bar \partial v(\zeta)|\leq C_2 \leq C_2e^{m\varphi_S(\xi)}, \qquad
\zeta=\xi+i\eta\in \C^n.
\end{equation}
By \cite{MagSigSigSno:2023}, Lemma {5.3}, there exists a constant $C_3>0$ such that 
for every $\zeta\in \C^n$, written as $\zeta=\xi+i\eta$ for $\xi, \eta\in \R^n$, we have
\begin{equation*}
  |v(\zeta)| \leq C_3\big(1+|e^\zeta|\big)^{a
  }
\sup_{w\in {\mathbb B}}e^{m\varphi_S(\xi+\Re\,  w)}
\leq C\big(1+|e^\zeta|\big)^{a
}e^{m\varphi_S(\xi)}, 
\end{equation*}
where $C=C_3\sup_{w\in {\mathbb B}}e^{m\varphi_S(\Re \, w)}$ and $\B$ is the open euclidean unit ball in $\C^n$.
We change the coordinates back to $z=e^\zeta$, use the fact that
$\Log\, z=\xi$ and conclude (\ref{u_uni_bound}). 
\end{proof} \bigskip

\medskip

We now see that an entire function  with a finite estimate of the form \eqref{eq:5.8} with a weight $\varphi$ that grows like $2mH_S+2\nu+a\log(1+|z|^2)$ is a member of $\P^S(\C^n)$, if $a$ is smaller than the distance $d_m$ between $mS$ and $\N^n\setminus mS$ in the $L^1$-norm. It suffices then also if $a$ is smaller than the euclidean distance $\operatorname{dist}(mS,\N^n\setminus mS)$. One difficulty is that $d_m$ is dependent on $m$ so it is not clear that one can choose a number $a$ that is smaller than this distance for all $m$ large enough. 

\medskip

The subset $S_m=\operatorname{ch}(S\cap(1/m)\N^n)$ of $S$ is such that the polynomial space $\P^{S_m}_m(\C^n)$ is the same as  $\P^{S}_m(\C^n)$. The set $mS_m$ has a larger distance than $mS$ to the next lattice point. Also $mS_m$ is an \emph{integral polytope}, meaning that its vertices are lattice points.  We can estimate the euclidean distance from a convex integral polytope to the next lattice point using methods from linear algebra. 

\begin{lemma}\label{integral_polytope}
    Let $P$ be a convex integral polytope with nonempty interior contained in some box $[0,M]^n$ where $M>0$. Then $$\operatorname{dist}(P, \Z^n\setminus P)\geq 1/(\sqrt n (n-1)! M^{n-1}).$$
\end{lemma}

\begin{proof}
    Let $P=\ch\{a_1,\dots, a_N\}$ with vertices $a_j\in \N^n$ for $j=1,\dots, N$.  The boundary of $P$ lies in the union of its boundary hyperplanes, each passing through some $(n+1)$-tuple of the vertices of $P$ in a general position. Since $P$ has nonempty interior such an $n$-tuple will exist.     
    
    Let $a_{\ell_1}, \dots , a_{\ell_n}$ be any such tuple and $A$ be the hyperplane through those points. Let $v_j= a_{\ell_j}-a_{\ell_1}$ for $j=2,\dots, n$. A normal $\eta=(\eta_1,\dots, \eta_n)$ to $A$ can be found with the coordinates $\eta_j= \det(e_j, v_2, \dots, v_n)$, where $e_j$ is the $j$-th unit vector. 

    For each vector $v_j$, $j=2,\dots, n$, each coordinate $v_{j,k}$, $k\in [n]$ is an integer with $|v_{j,k}|\leq M$. That implies that $\eta_j$ is an integer with $|\eta_j|\leq (n-1)! M^{n-1}$. Hence $|\eta|\leq \sqrt{n}(n-1)!M^{n-1}$.

    The distance from $x\in \Z^n\setminus P$ to $A$ is $|\scalar{\eta}{x}-\scalar{a_1}{x}|/|\eta|$. If $x\notin A$ then this distance is at least $1/|\eta|$. If $x\in A$ then the point of $P$ closest to $x$ is on some lower dimensional face of $P$. That face lies on the intersection of $A$ with some other boundary hyperplane $A'$ of $P$ with $x\notin A'$. The distance of $x$ from $P$ is therefore greater or equal to its distance to the boundary hyperplanes of $P$ that $x$ does not lie on, which is at least $ 1/(\sqrt{n}(n-1)! M^{n-1})$.\end{proof}
\bigskip

This lemma implies that the distance of $mS_m$ to the next lattice point grows slowly enough for our purposes. 

\begin{corollary}\label{dist^{1/m}}
    Let $S_m=\operatorname{ch}(S\cap (1/m)\Z^n)$. Then 
    $$\operatorname{dist}(mS_m, \Z^n\setminus (mS_m))^{1/m}\to 1 \quad \text{ as }m\to \infty.$$
\end{corollary}
\begin{proof}
    Clearly $\operatorname{dist}(mS_m, \Z^n\setminus (mS_m))\leq 1$ for all $m\in \N$. Each $mS_m$ is an integral polytope contained in $[0, m\varphi_S(\mathbf{1})]^n$, so by Lemma \ref{integral_polytope}, we have\\
    
    \hfill $\operatorname{dist}(mS_m, \Z^n\setminus (mS_m))^{1/m}\geq 1/(\sqrt{n}(n-1)!m\varphi_S(\mathbf{1}))^{1/m}\to 1. 
    $\hfill
\end{proof} 
\bigskip

A good candidate as a weight function $\varphi$ in \eqref{eq:5.8} is some plurisubharmonic function with the same growth as $2mH_{S_m}+2\nu$, in light of Proposition \ref{cor:u_uni_bound}, which is preferably large outside the set $K$. The function $V^S_K$ is therefore an ideal candidate as a weight. We will need to establish a positive lower bound on $V^S_K$ outside $K$, in other words that we should be able to escape one of its sublevel sets. A sublevelset of $V^S_{K}$ would be open  if $V^S_{K}$ were upper-semicontinuous, 
so we replace $K$ by a slightly fatter $S$-convex compact subset of $\Omega$ with that property.

\begin{lemma}\label{V_halfsamfellt}
    If $K$ is compact and $S$-convex, $\Omega\subseteq \C^{*n}$ is a neighborhood of $K$ and 
    $f\in \O(\Omega)$, then there is a $\delta>0$ such that the $S$-hull of $K_\delta=K+B(0,\delta)$ 
    is contained in $\Omega$, and $K'=\widehat{K}^S_\delta$ is such that $V^S_{K'}$ is continuous on $\C^{*n}$. 
\end{lemma}

\begin{proof}
     The function $\Phi^S_{K}$ is lower-semicontinuous for any choice of the set $K$, since it can be defined as the supremum of a family of continuous functions, by \cite{MagSigSigSno:2023}, Proposition 2.2. This implies that the sets $U_\delta= \set{z\in \C^{n}}{\Phi^S_{K_\delta}(z)>1}$ are open.  
     
     By \cite{MagSigSigSno:2023}, Proposition 4.8 (iii),  $\Phi^S_{K_\delta}\nearrow \Phi^S_K$ pointwise as $\delta\searrow 0$, which implies that $\bigcup_{\delta>0}U_\delta=\set{z\in \C^n}{\Phi^S_{K}(z)>1}= \C^n\setminus K$.  We may assume that $\Omega$ is bounded. Then there exists a compact 
    $B\subset \C^n$ that contains $\Omega$. Now $\bigcup_{\delta>0}U_\delta$ forms an increasing open cover of the compact set $B\setminus \Omega$, so there exists a $\delta>0$ such that $B\setminus \Omega\subseteq U_\delta$, which implies that $K':= \set{z\in \C^{n}}{\Phi^S_{K_\delta}(z)=1}\subset \Omega$. 
    
    Now $K'$ is an $S$-convex compact set and it is easy to see that $\Phi^S_{K'}= \Phi^S_{K_\delta}$. Furthermore, by \cite{MagSigSig:2023}, Theorem 1.1, $V^S_{K'}= V^S_{K_\delta}$ on $\C^{*n}$.  By \cite{MagSigSigSno:2023}, Lemma 5.2 and Propositions 5.3 and 5.4, $V^{S}_{K'}$ is continuous on $\C^{*n}$.
\end{proof}\\

We next show that a neighborhood of $K$ contains some sublevel set  of $V^{S*}_K$. Furthermore, if $V^{S*}_K|_K=0$, then the sublevel set produced in the following lemma is a neighborhood of $K$ and a sublevel set of $V^{S}_K$. 

\begin{lemma}\label{sublevelsetV}
    If $\Omega$ is a neighborhood of a $S$-convex compact set $K\subset\C^{n}$, then there is an $R>1$ such that $X_R=\set{z\in \C^n}{V^{S*}_K<\log R}$ is a relatively compact subset of $\Omega$.
\end{lemma}
\begin{proof}
We may assume that $\Omega$ is bounded. Let $B\subset\C^n$ be any compact set that contains $\Omega$. For every $z\in B\setminus \Omega$ there exists a $p_z\in \P^S_{m_z}(\C^n)$ for some $m_z\in \N$ with $\|p_z\|_K=1$ and with $|p_z(z)|>1$. Let $r_z$ be any number with $1<r_z<|p_z(z)|^{1/m_z}$. Since $B\setminus \Omega$ is compact there exist finitely many $z_1,\dots,z_N$ such that $p_j=p_{z_j}$, $m_j=m_{z_j}$ and $r_j=r_{z_j}$ satisfy 
$$B\setminus \Omega\subset \set{w\in \C^n}{|p_{j}(w)|^{1/m_j}>r_{j},\;j=1,\dots,N}.$$

Let $v= \max_{j=1,\dots, N} \log|p_{j}|^{1/m_{j}}$ and $R=\min\{r_{1},\dots,r_{N}\}$. Then $v|_{B\setminus \Omega}>R$ and since  $B\setminus \Omega$ is compact and $v$ is continuous it takes some lowest value $\min_{B\setminus \Omega}v>R$. Let $\min_{B\setminus \Omega}v>R'>R$.  The function 
$$u(z)=\begin{cases}
    \max\{\log R', v(z)
    \},\quad & z\in \C^n\setminus B\\
    v(z), \quad & z\in  B.\\
\end{cases}$$
is plurisubharmonic by the gluing theorem. Since $u\in \L^S(\C^n)$ and $u|_K\leq 0$ we can conclude that $X_R\subset \set{z\in \C^n}{u(z)\leq\log R}$, which is a closed bounded subset of $\Omega$. \\   
\textcolor{white}{.}
\end{proof}
\bigskip

We are now ready to start the proof of our main result.\\

\textbf{Proof of Theorem \ref{bigTheorem}:} Let $\Omega$ be the domain of $f$. By Lemma \ref{V_halfsamfellt} and \cite{MagSigSigSno:2023}, Theorem 1.1, we may assume that on $\C^{*n}$ we have $V^S_K$ continuous and  $V^S_K=\log\Phi^S_K$. By Lemma \ref{sublevelsetV} there is an $R>1$ such that  $X_R=\set{z\in \C^n}{V^{S}_K<\log R}$ is an open relatively compact subset of $\Omega\cap \C^{*n}$. 

Let $1<r<R$ and $0<\gamma<1$ be such that $\gamma^2>1/r$. Now $\Phi^{S}_{K,m}\nearrow \Phi^S_K$ uniformly on the bounded set $X_R\subset\C^{*n}$ as $m\to \infty$ by \cite{MagSigSigSno:2023}, Proposition 2.2, so for some $m_0\in \N$ we have $\log\Phi^{S}_{K,m}>\log \Phi^S_K-\log(1/\gamma)$
on $X_R$ for all $m\geq m_0$. By Corollary \ref{dist^{1/m}}, there is an $m_1\geq m_0$ such that $\operatorname{dist}(mS_m,\N^n\setminus mS_m)^{1/m}>1/2$ for all $m\geq m_1$. Now $V^S_K$ is continuous on $X_R$, so $X_{r}= \set{z\in \C^n}{V^{S}_K<\log r}$ is relatively compact in $X_R$. We may therefore  
take $\chi\in \mathcal C^\infty_0(X_{R})$ with $0\leq \chi\leq 1$,
and $\chi=1$ in $X_{r}$. Denote $V_m=\log\Phi^{S*}_{K,m}$ and $\nu(z)= \scalar{\mathbf{1}}{\Log\,z}$.

Define for every 
$z\in \C^n$  
\begin{align*}
     \psi_m(z)&=2mV_m(z)+2\nu(z)+ 2^{-m}\log(1+|z|^2), 
     \\
\eta_m(z)&=\psi_m(z)-2\log(1+|z|^2).
\end{align*}
Since $f$ is holomorphic in a neighborhood of $\supp \chi$, we have 
$\|f\bar\partial \chi\|_{\eta_m}<+\infty$ and $\bar\partial (f\bar\partial \chi)=0$. 

By Theorem \ref{thm:5.1},
there exists a solution 
$u_m\in \mathcal C^\infty(\C^n)$ of $\bar\partial u_m=f\bar\partial \chi$
satisfying
\begin{equation} \label{eq:3.3}
  \|u_m\|_{\psi_m}^2=\int_{\C^n}|u_m|^2(1+|z|^2)^{-1/{2^m}}
e^{-2mV_m(z)-2\nu(z)}\, d\lambda\\
\leq 2\|f\bar\partial \chi\|_{\eta_m}^2.
\end{equation}

 Now $V_m\leq V^{S_m*}_K \leq H_{S_m}+c_V$ for some constant $c_V\in \R$ by \cite{MagSigSigSno:2023}, Proposition 4.5. Therefore
\begin{equation}
  \label{eq:6.10}
\int_{\C^n} |u(z)|^2(1+|z|^2)^{-1/{2^m}}  
e^{-2mH_{S_m}(z)-2\nu(z)} d\lambda(z)
\leq e^{2mc_V}\|u\|_{\psi_{m}}^2<+\infty.
\end{equation}
We let $p_m=f\chi-u_m\in \O(\C^n)$.
By Corollary \ref{cor:u_uni_bound},   $|p_m(z)| \leq C(1+|z|)^{1/{2^m}} e^{mH_{S_m}(z)}$ for all $z\in \C^n$ for some constant $C>0$.  Theorem 3.6
from \cite{MagSigSigSno:2023} then implies that
$p_m\in {\mathcal P}^{S_m}_m(\C^n)=  {\mathcal P}^{S}_m(\C^n)$ for all $m\geq m_1$. Our goal is to show that $\|u_m\|_K= \|f-p_m\|_K$ can be made arbitrarily small by taking $m\geq m_1$ large enough.

Denote the mean of $v\in L^1_{\operatorname{loc}}(\C^n)$ 
over a ball $B(z,\delta)$ with center $z$ and radius $\delta>0$ by 
$$
{\mathcal M}_\delta v(z)=\dfrac 1{\Omega_{2n}\delta ^{2n}}\int_{B(z,\delta)}v\, d\lambda
$$
where $\Omega_{2n}$ is the volume of the unit ball in $\mathbb{R}^{2n}$. Since $u_m\in \O(\C^n \setminus \operatorname{supp}
\bar\partial \chi)$, the mean value theorem implies that for every 
$z\in \C^n$  we have
$u_m(z)={\mathcal M}_\delta u_m(z)$
and consequently, by the Cauchy-Schwarz inequality
\begin{align}\label{eq:3.6}
  |u_m(z)|&\leq \Omega_{2n}^{-1}\delta^{-2n} \int_{B(z,\delta)}|u_m| \,
            d\lambda \nonumber\\
&=\Omega_{2n}^{-1}\delta^{-2n} 
\int_{B(z,\delta)}|u_m|e^{-\psi_m/2}\cdot e^{\psi_m/2} \, d\lambda \nonumber \\
&\leq \Omega_{2n}^{-1}\delta^{-2n} \|u_m\|_{\psi_m}
\Big(\int_{B(z,\delta)} e^{\psi_m} d\lambda \Big)^{1/2}\nonumber \\
&\leq  \Omega_{2n}^{-1/2}\delta^{-n} \sqrt{2} 
\|f\bar\partial \chi\|_{\eta_m}\Big({\mathcal M}_\delta
\big(e^{\psi_m}\big)(z)\Big)^{1/2}.
\end{align}

Let $0<\varepsilon<\log (\gamma r)$ and denote $Y=\operatorname{supp} \bar\partial\chi$. Let $0<\delta<\operatorname{dist}(K,Y)$ be such that $V^{S}_K(z+w)\leq  \varepsilon$ for all $z\in K$ and 
 all $w\in B(0,\delta)$. 
 Furthermore $$V_m(\zeta)\leq V^{S}_K(\zeta)\leq \varepsilon, \qquad z\in K, \, \zeta \in B(z,\delta).$$ 

Since the set $\overline X_r\subset \C^{*n}$ is compact, we have  $B= \max_{z\in \overline X_r}|z|<+\infty$ and 
$b=\min_{z\in \overline{X}_r} \nu(z)>-\infty.$ We have 
\begin{align}\label{eq:3.8}
  {\mathcal M}_\delta\big(e^{\psi_m}\big)(z)
  &=\dfrac 1{\Omega_{2n}\delta^{2n}}
    \int_{B(z,\delta)} (1+|\zeta|^2)^{1/{2^m}} e^{2\nu(z)}  e^{2mV_m(\zeta)}  \, d\lambda(\zeta) \nonumber\\
&\leq (1+B^2)^{1/{2^m}}e^{2n\log B} e^{2m\varepsilon}. 
\end{align}

Since $V_m\geq \log(\gamma r)$ holds on $\C^n\setminus X_r$ for all $m\geq m_1$, we observe next that
\begin{align}
    \label{eq:3.7}
  \|f\bar\partial \chi\|_{\eta_m}
&= \Big(\int_{Y}|f\bar\partial \chi|^2(1+|\zeta|^2)^{2-1/{2^m}}e^{-2mV_m(z)-2\nu(z)}
\, d\lambda \Big)^{1/2}\nonumber\\
&\leq \Big(\int_{Y}|\bar\partial \chi|^2
\, d\lambda \Big)^{1/2}\cdot (1+B^2)^{2-1/{2^m}}e^{-2b}\cdot \frac{\|f\|_{X_r}}{(\gamma r)^{m}}. 
\end{align}

Denoting $c_\chi= \int_{Y}|\bar\partial \chi|^2
\, d\lambda<+\infty$, we see 
that \eqref{eq:3.6}, \eqref{eq:3.8} and \eqref{eq:3.7} together imply that 
$$|u_m(z)|\leq \Omega_{2n}^{-1/2}\delta^{-n} \sqrt{2} 
c_\chi (1+B^2)^2 e^{n\log B-2b}\|f\|_{X_r} \left(\frac{e^\varepsilon}{\gamma r}\right)^m.$$
The constants $\delta, B,b$ and $c_\chi$ were chosen independent of $m$, and $|e^\varepsilon/\gamma r|<1$. Therefore $\|u_m\|_K$ can be made arbitrarily small by choosing $m$ large enough. We conclude that  $p_m$ tend to $f$ uniformly on $K$ as $m$ tends to infinity.
\hfill $\square$

{\small 
\bibliographystyle{siam}
\bibliography{rs_bibref}

\noindent
Science Institute,
University of Iceland,
IS-107 Reykjav\'ik,
ICELAND \\
alfheidur@hi.is
}

\end{document}